\documentclass[a4paper,12pt]{article}

\usepackage{graphicx}%
\usepackage{multirow}%
\usepackage{amsmath,amssymb,amsfonts}%
\usepackage{amsthm}%
\usepackage{mathrsfs}%
\usepackage[title]{appendix}%
\usepackage{xcolor}%
\usepackage{textcomp}%
\usepackage{manyfoot}%
\usepackage{booktabs}%
\usepackage{algorithm}%
\usepackage{algorithmicx}%
\usepackage{algpseudocode}%
\usepackage{listings}%

\usepackage{tikz}
\usetikzlibrary{arrows,decorations.markings}
\usetikzlibrary{arrows.meta}

\newtheorem{theorem}{Theorem}[section]
\newtheorem{lemma}[theorem]{Lemma}
\newtheorem{cor}{Corollary}[theorem]

\begin{document}
\title{On the minimum number of entries in a pair of maximal orthogonal partial Latin squares}

\author{Diane M. Donovan\footnote{Donovan acknowledges the support of the Australian Government through funding of the Australian Research Council  Centre of Excellence for Plant Success in Nature \& Agriculture (project number CE200100015).},\\
\normalsize{ARC Centre of Excellence for Plant Success in Nature and Agriculture,}\\
\normalsize{School of Mathematics and Physics, University of Queensland,}\\
\normalsize{Brisbane 4072, Australia.} \\
\texttt{(dmd@maths.uq.edu.au)}\\
\\
Mike Grannell,\\
\normalsize{School of Mathematics and Statistics, The Open University,} \\
\normalsize{Walton Hall, Milton Keynes MK7 6AA, United Kingdom.}\\
\texttt{(m.j.grannell@open.ac.uk)}\\
\\
Emine \c{S}ule Yaz{\i}c{\i}\footnote{Yaz{\i}c{\i} acknowledges the support of the Turkish Government through funding by The Scientific and Technological Research Council of Turkey (TUBITAK grant number: 121F111)},\\
\normalsize{Department of Mathematics, Ko\c{c} University,}\\
\normalsize{Sar{\i}yer, 34450, \.{I}stanbul, Turkey.}\\
\normalsize{Also}\\
\normalsize{School of Mathematics and Physics, University of Queensland,}\\
\normalsize{Brisbane 4072, Australia.} \\
\texttt{(eyazici@ku.edu.tr)}\\
}

\maketitle

\newpage



\begin{abstract}
It is shown that if $F$ denotes the number of filled cells in a superimposed pair of maximal orthogonal partial Latin squares of order $n$, then $F\ge n^2/3$. This resolves a conjecture raised in an earlier paper by the current authors. It is also shown that, for $n\ge 21$, the least possible number of filled cells in a pair of maximal orthogonal partial Latin squares is $\lceil n^2/3 \rceil$, and that the structure that achieves this bound is unique up to permutations of rows, columns and entries.
\end{abstract}

\section{Introduction} \label{sec:intro}
This paper follows on from an earlier paper \cite{MOPLS1} by the current authors and resolves one of the conjectures raised in that paper.

We start by recalling some basic definitions. A \emph{Latin square} of order $n$, denoted by LS($n$), is an $n\times n$ array $L=[L(i,j)]$, with entries from an $n$-element set, such that each element of this set occurs once in every row and once in every column of $L$. If convenient, rows and columns can be indexed by $[n]=\{1,2,\ldots,n\}$, and the set of entries can also be taken as $[n]$. An alternative representation of such a Latin square is as a \emph{transversal design}, TD$(3,n)$, which comprises a set of $3n$ points partitioned into three $n$-element subsets, called groups, and a set of $n^2$ triples such that each pair of points from different groups appears in precisely one triple and no triple contains more than one point from each group. The correspondence between these two representations is by taking the three groups to represent respectively rows, columns and entries, and a triple $(a,b,c)$ to represent entry $c$ in the cell corresponding to the intersection of row $a$ and column $b$. An LS($n$) can also be viewed as a partition of the complete tripartite graph $K_{n,n,n}$ into copies of $K_3$. The three parts of the graph represent the rows, columns and entries and the copies of $K_3$ form the (row, column, entry) triples. An advantage of thinking in these terms is that the three vertex parts clearly have equal status.

A \emph{partial} Latin square of order $n$, denoted by PLS($n$), is defined in the same manner as a Latin square of order $n$ except that some of the cells may be empty, in other words, each element of the entry set occurs \emph{at most} once in every row and \emph{at most} once in every column. Alternative representations are as a partial transversal design PTD($3,n$) or as a decomposition of a subgraph of $K_{n,n,n}$ into copies of $K_3$.

A \emph{maximal} PLS($n$) is a partial Latin square of order $n$ that cannot be extended to another PLS($n$) by inserting any entry into any empty cell. A maximal PLS($n$) is denoted by MPLS($n$). Horak and Rosa \cite{HorakRosa} proved that a PSL($n$) with less than $n^2/2$ filled cells cannot be maximal. On the other hand, for each positive integer $n$ there is an MPLS($n$) with the number of filled cells equal to $\lceil n^2/2 \rceil$. Figure \ref{fig:MPLS6and7} shows squares of orders $n=6$ and $n=7$, and it is easy to see how these generalise.
\begin{figure}[ht]
\[ \begin{array}{|rrrrrr|}\hline ~1&~2&~3&-&-&-\\3&1&2&-&-&-\\2&3&1&-&-&-\\-&-&-&~4&~5&~6\\
-&-&-&6&4&5\\-&-&-&5&6&4\\ \hline \end{array}\hspace{2cm}
\begin{array}{|rrrrrrr|}\hline ~1&~2&~3&-&-&-&-\\3&1&2&-&-&-&-\\2&3&1&-&-&-&-\\-&-&-&~4&~5&~6&~7\\
-&-&-&7&4&5&6\\ -&-&-&6&7&4&5\\ -&-&-&5&6&7&4\\\hline \end{array}\]
\caption{MPLS(6) and MPLS(7)} \label{fig:MPLS6and7}
\end{figure}

By permuting rows and columns, any MPLS($n$) with $n>1$ and having $\lceil n^2/2 \rceil$ filled cells can be put into a form similar to those shown in Figure \ref{fig:MPLS6and7}. We give a short  proof of this because it will help in our subsequent discussions.

\begin{lemma} \label{lem:HR} Suppose that $n>1$ and that $M$ is an MPLS($n$) of minimum size, so that $M$ has $F=\lceil n^2/2 \rceil$ filled cells. Then $M$ can be partitioned into two Latin squares, on different point sets, with orders that sum to $n$ and are either $\lfloor n/2 \rfloor$ or $\lfloor n/2 \rfloor+1$.
\end{lemma}
\begin{proof}
If we put $s=\lfloor n/2 \rfloor$, then $n=2s$ or $2s+1$, so that $M$ has $F=2s^2$ or $2s^2+2s+1$ filled cells, respectively. Let $m$ be the minimum frequency of $M$. By this we mean that $m$ is the minimum of the number of entries in any row or any column and the number of occurrences of any entry. Then $m$ is at most the average frequency, namely $F/n$, which gives $m\le s$ in both cases. By the symmetry between rows, columns and entries, we may therefore assume that there is some entry, say $y_1$, that appears in $M$ exactly $m$ times, and $M$ may then be rearranged into the form shown in Figure \ref{fig:HR}, where $M$ is subdivided onto four regions labelled $M_1,M_2,M_3$ and $M_4$. The remaining entries may be taken as $y_2,y_3,\ldots,y_n$.

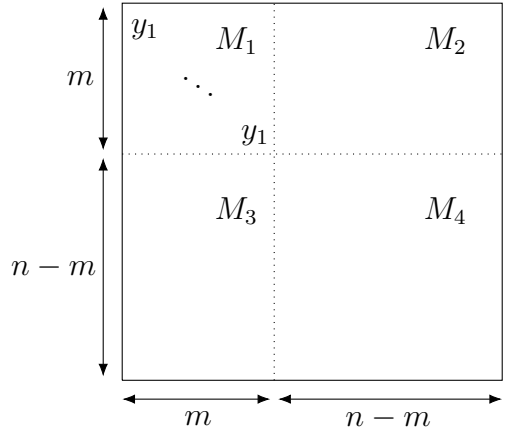
\begin{figure}[!ht]\begin{center}
\begin{tikzpicture}[x=0.5mm,y=0.5mm]

\draw (10,10)--(10,110)--(110,110)--(110,10)--(10,10);
\draw [dotted] (10,70)--(110,70);
\draw [dotted] (50,10)--(50,110);

\node at (40,100) {$M_1$};
\node at (95,100) {$M_2$};
\node at (40,55) {$M_3$};
\node at (95,55) {$M_4$};

\node at (16,103) {$y_1$};
\node at (30,90) {$\ddots$};
\node at (45,75) {$y_1$};

\draw [Latex-Latex] (10,5)--(49,5);
\node at (30,0) {$m$};
\draw [Latex-Latex] (51,5)--(110,5);
\node at (80,0) {$n-m$};

\draw [Latex-Latex] (5,71)--(5,109);
\node at (-1,90) {$m$};
\draw [Latex-Latex] (5,11)--(5,69);
\node at (-8,40) {$n-m$};
\end{tikzpicture}
\end{center}
\caption{A partition of $M$.\label{fig:HR}}
\end{figure}

Region $M_1$ contains all occurrences of $y_1$ and so region $M_4$ must be full, otherwise $y_1$ could be inserted into an empty cell contradicting the maximality of $M$. Since $m$ is the minimal frequency we then see that $M$ has $m$ rows each with at least $m$ entries and $n-m$ rows each with at least $n-m$ entries, giving $F\ge m^2+(n-m)^2=2m^2-2mn +n^2$. If $n=2s$ this gives $2s^2\ge 2m^2-4ms+4s^2=2s^2+2(s-m)^2$, which is only possible if $m=s$. Similarly, if $n=2s+1$, we get $2s^2+2s+1\ge 2m^2-4ms-2m+4s^2+4s+1$, and this gives $2s^2\ge 2s^2+2(s-m)^2+2(s-m)$, which again is only possible if $m=s$. In either case $F=m^2+(n-m)^2$, and it then follows that $M$ has $m$ rows with exactly $m$ entries and  $n-m$ rows with exactly $n-m$ entries. By symmetry between rows and columns, the same must be true for columns. Consequently regions $M_2$ and $M_3$ must be empty and regions $M_1$ and $M_4$ must be full.

Let $S$ denote the set of $n-m$ entries occurring in column $m+1$ of $M$, all of which occur in region $M_4$. We may take $S=\{y_{m+1},y_{m+2},\ldots, y_n\}$. Since $M$ is maximal, each row $i$ of $M$ for $1\le i\le m$ must contain all the entries in $S'=\{y_1,y_2,\ldots,y_m\}$, otherwise there exists some $k\le m$ such that the entry $y_k$ can be placed in cell $(i,m+1)$. Consequently each element of $S'$ must occur in each of the rows $1,2,\ldots,m$ of $M$ and so $M_1$ is a Latin square of order $m$ with entries from $S'$. Next consider the set of entries occurring in row $1$ of $M$. By employing the same argument again it can be seen that each column $m+i$ of $M$ for $1\le i \le n-m$ must contain all the entries in $S$. Hence $M_4$ is a Latin square of order $n-m$ with entries from $S$.
\end{proof}

\newpage

Two Latin squares of the same order $n$ are said to be \emph{orthogonal} if, when superimposed, each of the $n^2$ possible ordered pair of entries appears once and only once in the cells of the superimposed square. Such a pair of orthogonal Latin squares, when presented as two separate squares, is usually referred to as a pair of MOLS($n$) (``Mutually Orthogonal Latin Squares of order $n$''). We will generally use the superimposed form and consider a pair of orthogonal Latin squares as a single array with each cell having a first and a second entry. We denote such an array as an OLS($n$). An alternative representation is as a transversal design TD$(4,n)$, which comprises a set of $4n$ points partitioned into four $n$-element subsets, again called groups, and a set of $n^2$ quadruples such that each pair of points from different groups appears in precisely one quadruple and no quadruple contains more than one point from each group. The correspondence between these two representations is by taking the four groups to represent respectively rows, columns, first entries and second entries, and a quadruple $(a,b,c,d)$ to represent first entry $c$ and second entry $d$ in the cell corresponding to the intersection of row $a$ and column $b$ of the superimposed squares. An OLS($n$) can also be viewed as a partition of the complete 4-partite graph $K_{n,n,n,n}=K_{4\times n}$ into copies of $K_4$. The four parts of the graph represent the rows, columns, first entries and second entries, and the copies of $K_4$ form the (row, column, first entry, second entry) quadruples.

Two partial Latin squares of the same order $n$, having the same filled cells as one another, are said to be \emph{orthogonal} if, when superimposed, each of the $n^2$ possible ordered pair of entries appears at most once in the cells of the superimposed square. Again, we generally use the superimposed form and denote such an array as an OPLS($n$) (Orthogonal Partial Latin Square of order $n$). Alternative representations are as a partial transversal design PTD($4,n$) or as a decomposition of a subgraph of $K_{4\times n}$ into copies of $K_4$. The latter aspect is discussed in more detail after the definition in the next paragraph.

A \emph{maximal} OPLS($n$) is an OPLS($n$) that cannot be extended to another OPLS($n$) by inserting any entry pair into any empty cell. This is analogous to a maximal PLS($n$). We denote a maximal OPLS($n$) as an MOPLS($n$).  With the exceptions of $n=2$ and $n=6$, the maximum size (number of filled cells) of an MOPLS($n$) is $n^2$ since there is a pair of orthogonal Latin squares of order $n$ for every positive integer $n\ne2,6$. In the previous paper \cite{MOPLS1} a considerable number of partial results led to the conjecture that an OPLS($n$) with less than $n^2/3$ filled cells cannot be maximal.

\bigskip

An OPLS($n$) can be considered as a decomposition of a subgraph $G$ of $K_{4\times n}$ by taking the vertex set to be $V_1\cup V_2\cup V_3\cup V_4$, where $V_1=[n]\times\{1\}$, \linebreak $V_2=[n]\times\{2\}$, $V_3=[n]\times\{3\}$, and  $V_4=[n]\times\{4\}$, represent respectively the rows, columns, first entries, and second entries of the OPLS($n$). The edges of $G$ are determined by entries in the OPLS($n$): if the cell $(a,b)$ contains the entry pair $(c,d)$ then $G$ contains  a copy of $K_4$ with vertex set $\{(a,1),(b,2),(c,3),(d,4)\}$, and vice-versa. Thus $G$ decomposes into copies of $K_4$. As remarked earlier, this viewpoint shows that the vertex sets $V_i$ have equal status. So if no special property is assumed for any one of the vertex sets, then any results for any one of them will, by symmetry, apply to the others.

If the given OPLS($n$) is maximal (i.e. it is an MOPLS($n$)) then the corresponding complementary graph $G'=K_{4\times n}-G$ is $K_4$-free because, if there were a copy of $K_4$ on the vertex set $\{(\alpha,1),(\beta,2),(\gamma,3),(\epsilon,4)\}$ in $G'$ then cell $(\alpha,\beta)$ in the MOPLS($n$) could be filled with the entry pair $(\gamma,\epsilon)$, contradicting its maximality. The problem of finding the minimum size (number of filled cells) of an MOPLS($n$) is therefore equivalent to finding the maximum size of the corresponding $K_4$-free complementary graph $G'$. If there are $F$ filled cells in an OPLS($n$) then the number of edges between $V_i$ and $V_j$ for $1\le i,j\le 4$ in the corresponding complementary graph $G'$ is $n^2-F$, so that $G'$ has $6(n^2-F)$ edges. Furthermore, given a vertex $(x,i)$ with degree $v$, the number of edges joining $(x,i)$ to vertices in $V_j$ ~$(j\ne i)$ is $v/3$.

Determination of the  maximum number of edges in a $K_4$-free multipartite graph is a problem that has been studied extensively in the literature.  In \cite{Bollobas} the authors show that, in particular, the number of edges in a  $K_4$-free $4$-partite graph with parts of size $n$ is at most $5n^2$. Jin \cite{Jin} proved that if $G$ is a $4$-partite graph with parts of size $n$ then $G$ contains a $K_4$ if the minimum vertex degree $\delta(G)$ satisfies $\delta(G)>\lfloor(2+\frac{1}{3})n\rfloor$. Maybe the most important asymptotic contribution to the area is due to Pfender \cite{Pfender} who considers $l$-partite graphs satisfying a certain density condition. For such a graph $G=(V_1\cup V_2\cup ...V_l, E)$, define the density $d(V_i,V_j)$ between the vertex parts $V_i$ and $V_j$ ($i\ne j$) as
\[d(V_i, V_j) = \frac {||G[V_i \cup V_j ]||}{|V_i|\cdot|V_j |},\]
where $||G[V_i \cup V_j ]||$ is the number of edges between $V_i$ and $V_j$, and $|V_k|$ is the number of vertices in $V_k$. Pfender proves that for large enough $l$, and apart from one specified family of graphs, such a graph having $d(V_i, V_j)\ge (k-2)/(k-1)$ for all $i\ne j$ must contain a copy of $K_k$. However $l$ depends on $k$ and the assumption in the paper \cite{Pfender} is that $l$ is very much larger than $k$. The smallest value of $l$ for which the theorem remains true has not generally been determined. In the current paper we are dealing with a particular instance of the case $l=k=4$. For the complementary graph $G'$ corresponding to an OPLS($n$) with $F$ filled cells, $d(V_i, V_j)=(n^2-F)/n^2$ and we will show that if $F<n^2/3$, so that $d(V_i, V_j)>2/3=(k-2)/(k-1)$, then $G'$ contains a copy of $K_4$ and the OPLS($n$) is not maximal.

\bigskip

It is always possible to find an MOPLS($n$) with around $n^2/3$ filled cells. Figure \ref{fig:MOPLS(9)} shows an example of an MOPLS(9). It is easy to explain how this can be generalised to larger orders $n\ge 21$.
\begin{figure}[ht]
\begin{center}
\[\begin{array}{|ccccccccc|}\hline
11&22&33&-&-&-&-&-&-\\ 23&31&12&-&-&-&-&-&-\\ 32&13&21&-&-&-&-&-&-\\
-&-&-&44&55&66&-&-&-\\ -&-&-&56&64&45&-&-&-\\ -&-&-&65&46&54&-&-&-\\
-&-&-&-&-&-&77&88&99\\ -&-&-&-&-&-&89&97&78\\ -&-&-&-&-&-&98&79&87\\
\hline
\end{array}\]
\end{center}
\caption{An MOPLS(9) with 27 filled cells.\label{fig:MOPLS(9)}}
\end{figure}
To check that the OPLS(9) shown in Figure \ref{fig:MOPLS(9)} is maximal, consider the empty cell in the first row and fourth column. It is clear that the only candidates for first entry are $7, 8, 9$. The same goes for candidates for second entry. But every ordered pair from $\{7,8,9\}$ already appears in the square. A similar argument applies to every empty cell.

For $n\ge 21$ a similar MOPLS($n$) may be created having $\lceil n^2/3 \rceil$ filled cells. Suppose that $n=3s+r$, where $r=0,1$ or 2. Starting from an empty $n\times n$ array, place three orthogonal Latin squares down the leading diagonal. The first (upper left) is of order $s$ with first and second entry symbols chosen from $\{1,2,\ldots,s\}$. The second (central) is of order $s$ with first and second entry symbols chosen from $\{s+1,s+2,\ldots,2s\}$ if $r=0$ or if $r=1$, or is of order $s+1$ with first and second entry symbols chosen from $\{s+1,s+2,\ldots,2s+1\}$ if $r=2$. The third (lower right) is of order $s$ if $r=0$, or is of order $s+1$ if $r=1$ or 2. The entry set for this third square comprises the remaining symbols from $\{1,2,\ldots,n\}$, not already used in the first and second squares. For $m\ne 2, 6$ there exists an OLS($m$)(see \cite{HBK}). Hence the construction is always possible for $n\ge 21$ since this entails $s\ge 7$. The argument that the resulting OPLS($n$) is maximal mirrors that given for the MOPLS(9). The number of filled cells is $F$ where
\[F=\left\{\begin{array}{ccccc}
3s^2&=&n^2/3 &\mbox{if} &r=0,\\
2s^2+(s+1)^2&=&(n^2 + 2)/3 &\mbox{if} &r=1,\\
s^2+2(s+1)^2&=&(n^2+2)/3 &\mbox{if} &r=2.
\end{array}\right\}=\left\lceil \frac{n^2}{3} \right\rceil. \]
We will prove in the next section that these values are best (i.e. minimum)  possible.

Finally in this section we define some additional terminology and notation. For any array we use $f_r(w)$ to denote the frequency of row $w$ (i.e., the number of filled cells in row $w$) and $f_c(x)$ to denote the frequency of column $x$. For an OPLS($n$) we will also speak of the frequency of a first entry $y$ and of a second entry $z$, meaning the number of times $y$ and $z$ appear as entries in the OPLS($n$).

\section{The main result}
We will prove the following result.

\begin{theorem}\label{th:main}
 Suppose that $M$ is an MOPLS($n$) with $F$ filled cells. Then $F\ge n^2/3$, with equality if and and only if every frequency (rows, columns, first entries, second entries) is $n/3$.
\end{theorem}

The first step is to establish the following lemma.
\begin{lemma}\label{lem:transversals}
Suppose that $D$ is a $d\times d$ array in which some cells may be filled and some may be empty. Let $T$ be a maximum partial transversal of $t$ empty cells in $D$, where $0\le t\le d$. By permuting rows and columns we can assume that $T=\{(1,1), (2,2),\ldots, (t,t)\}$, and let $T'$ be the continuation of this leading diagonal, comprising cells $(i,i)$ for $t+1\le i\le d$. Then every cell $(i,j)$ with $i,j\in\{t+1, t+2, \ldots,d\}$ is filled and $f_r(i)+f_c(j)\ge 2d-t$. In particular, for each cell $(i,i)\in T'$, $f_r(i)+f_c(i)\ge 2d-t$.
\end{lemma}
\begin{proof} By saying that $T$ is a maximum partial transversal of empty cells in $D$ we mean that no row or column of $D$ has more than one cell from $T$, and that there is no larger partial transversal of empty cells. We do not exclude the possibilities that $T$ is empty or that $T$ is a complete transversal of $d$ cells.

The array $D$ is illustrated in Figure \ref{fig:D}. It is partitioned into four parts labelled $D_1, D_2, D_3, D_4$ as shown. We use $\emptyset$ to denote an empty cell and $\star$ to denote a filled cell. Initially we assume that $t\ne 0,d$.
 \begin{figure}[!ht]\begin{center}
\begin{tikzpicture}[x=1mm,y=1mm]

\draw (10,10)--(10,110)--(110,110)--(110,10)--(10,10);
\draw [dotted] (10,70)--(110,70);
\draw [dotted] (50,10)--(50,110);

\node at (30,105) {$D_1$};
\node at (85,105) {$D_2$};
\node at (30,55) {$D_3$};
\node at (85,55) {$D_4$};

\node at (14,105) {$\emptyset$};
\node at (30,90) {$T$};
\node at (47,73) {$\emptyset$};
\draw [Latex-] (20,100)--(25,95);
\draw [-Latex] (35,85)--(40,80);

\node at (54,65) {$\star$};
\node at (80,40) {$T'$};
\node at (107,13) {$\star$};
\draw [Latex-] (60,60)--(75,45);
\draw [-Latex] (85,35)--(100,20);

\draw [Latex-Latex] (10,8)--(49,8);
\node at (30,3) {$t$};
\draw [Latex-Latex] (51,8)--(110,8);
\node at (80,3) {$d-t$};

\draw [Latex-Latex] (7,71)--(7,109);
\node at (0,90) {$t$};
\draw [Latex-Latex] (7,11)--(7,69);
\node at (0,40) {$d-t$};
\end{tikzpicture}
\end{center}
\caption{A partition of $D$.\label{fig:D}}
\end{figure}

The first observation is that region $D_4$ cannot contain any empty cells since otherwise $T$ would not be a maximum partial transversal of empty cells. So every row and column in the sub-array $D_4$ has $d-t$ filled cells.

Now consider any cell $(i,j)\in D_4$, so that $t+1\le i,j\le d$, and take $k\in\{1,2,\ldots,t\}$. If both cells $(i,k)$ and $(k,j)$ (in regions $D_3$ and $D_2$ respectively) were empty then $T$ could be extended by removing cell $(k,k)$ and adding cells $(i,k)$ and $(k,j)$, giving a transversal $\overline{T}$ having $t+1$ empty cells, a contradiction. So at least one of the cells $(i,k)$ and $(k,j)$ must be filled for each $k\in\{1,2,\ldots,t\}$. Thus the union of row $i$ in region $D_3$ and column $j$ in region $D_2$ contains at least $t$ filled cells. Since all the cells in region $D_4$ are filled, we obtain
\[f_r(i)+f_c(j)\ge t+2(d-t)=2d-t. \]
In particular, this applies when the cell $(i,j))$ lies in $T'$, i.e. when $i=j$.

In the extreme case $t=0$, region $D_4$ comprises the whole of the array $D$ and all cells must be filled. At the other extreme $t=d$, region $D_1$ comprises the whole of the array $D$ and there is nothing to prove.
\end{proof}

\bigskip

\begin{proof}[\textbf{Proof of Theorem \ref{th:main}}]
Let $m$ denote the minimum frequency in $M$ of rows, columns, first entries and second entries. Then $m\le F/n$. By symmetry we can assume, without loss of generality, that the frequency $m$ is achieved by a first entry $y_1$ and that $y_1$ appears paired with $m$ second entries that we can denote as $z_1,z_2,\ldots,z_m$. If $m>0$, the rows and columns of $M$ can be rearranged as shown in Figure \ref{fig:ABCD}, thereby partitioning $M$ into four regions labelled $A,B,C,D$. If $m=0$ then regions $A,B,C$ do not exist, and region $D$ comprises the whole of $M$.

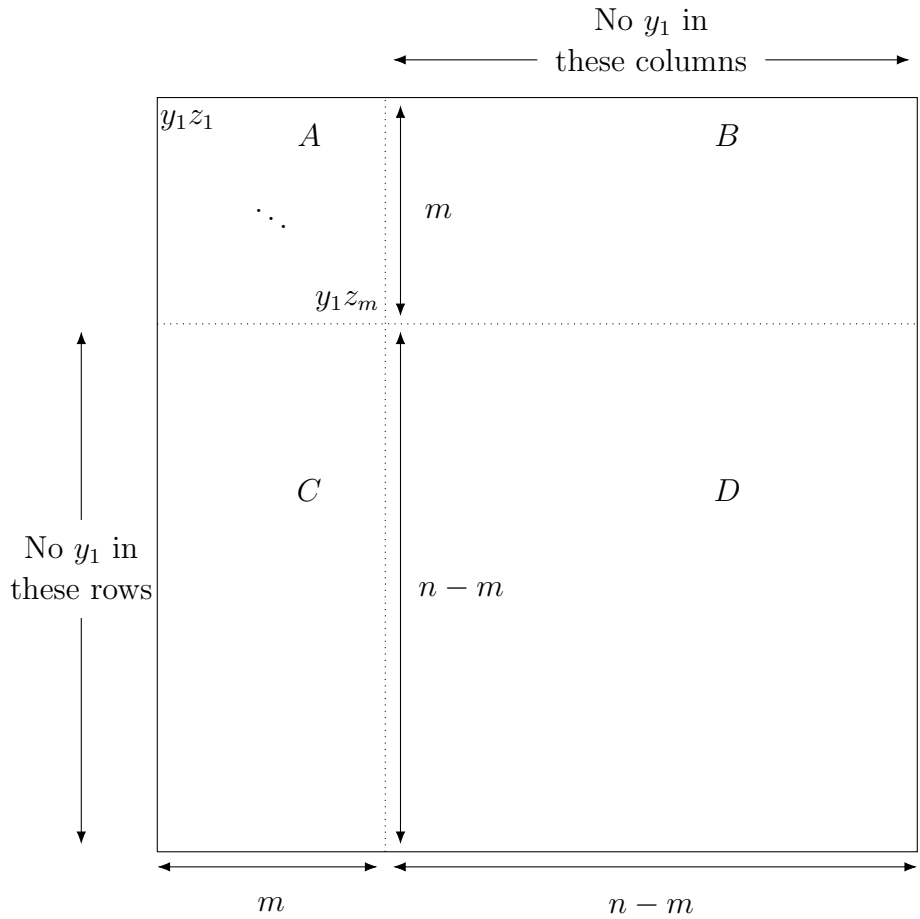
\begin{figure}[!ht]\begin{center}
\begin{tikzpicture}[x=1mm,y=1mm]

\draw (10,10)--(10,110)--(110,110)--(110,10)--(10,10);
\draw [dotted] (10,80)--(110,80);
\draw [dotted] (40,10)--(40,110);

\node at (30,105) {$A$};
\node at (85,105) {$B$};
\node at (30,58) {$C$};
\node at (85,58) {$D$};

\node at (14,107) {$y_1z_1$};
\node at (25,95) {$\ddots$};
\node at (35,83) {$y_1z_m$};
\node at (75,120) {No $y_1$ in};
\node at (75,115) {these columns};
\draw [Latex-] (41,115)--(60,115);
\draw [-Latex] (90,115)--(109,115);
\node at (0,50) {No $y_1$ in};
\node at (0,45) {these rows};
\draw [Latex-] (0,11)--(0,41);
\draw [-Latex] (0,54)--(0,79);

\draw [Latex-Latex] (10,8)--(39,8);
\node at (25,3) {$m$};
\draw [Latex-Latex] (41,8)--(110,8);
\node at (75,3) {$n-m$};

\draw [Latex-Latex] (42,81)--(42,109);
\node at (47,95) {$m$};
\draw [Latex-Latex] (42,11)--(42,79);
\node at (50,45) {$n-m$};
\end{tikzpicture}
\end{center}
\caption{A partition of $M$.\label{fig:ABCD}}
\end{figure}

If $I=(i,i)$ is one of the $m$ cells on the leading diagonal in region $A$, then $f_r(i)\ge m$ and $f_c(i)\ge m$, so
$f_r(i)+f_c(i)\ge 2m$.

Region $D$ will contain some filled cells and some empty cells. We may apply Lemma \ref{lem:transversals} to region $D$ by taking  $d=n-m$.

If $I=(i,i)$ is one of the $t$ empty cells in the maximum partial transversal $T$ of empty cells in $D$ then the union of row $i$ and column $i$ must contain at least $n-m$ entries, otherwise there would be some second entry $z_k$ with $k\ne 1,2,\ldots,m$ that  could be paired with $y_1$ and placed in $I$, contradicting the maximality of $M$. Hence $f_r(i)+f_c(i)\ge n-m$ for each $i\in\{m+1,$ $m+2,\ldots,m+t\}$.

If $I=(i,i)$ is one of the $n-m-t$ filled cells in region $D$ lying in the continuation partial transversal $T'$ then, by Lemma \ref{lem:transversals} $f_r(i)+f_c(i)\ge 2(n-m)-t$. As an aside we remark that $f_r(i)+f_c(j) \ge 2(n-m)-t$ for any cell $(i,j)$ lying in the region corresponding to $D_4$ of Figure \ref{fig:D}.

Summing over all the cells on the leading diagonal gives
\[\sum_{i=1}^n \left(f_r(i)+f_c(i)\right)\ge 2m^2+t(n-m)+(n-m-t)(2(n-m)-t).\]

Since
\[\sum_{i=1}^n f_r(i)=F=\sum_{i=1}^n f_c(i),\]
we obtain
\begin{align*}2F &\ge 2m^2+t(n-m)+(n-m-t)(2(n-m)-t)\\
 &=2m^2+(n-m)^2+(n-m-t)^2\\
 &=(n-m-t)^2+\frac{2n^2}{3}+\frac{n^2}{3}-2nm+3m^2\\
 &=(n-m-t)^2+\frac{(n-3m)^2}{3}+\frac{2n^2}{3}.
\end{align*}
This gives
\begin{align}F&\ge \left\lceil\frac{(n-m-t)^2}{2}+\frac{(n-3m)^2}{6}+\frac{n^2}{3}\right\rceil.\label{ineq:main}\end{align}

It follows that $F\ge n^2/3$. From the inequality (\ref{ineq:main}) it is clear that the equality $F=n^2/3$ is only possible if $t=n-m$, and $m=n/3$, which of course implies that $n\equiv 0$ (mod 3). And if $m=n/3$ then the frequency of every row (also every column, first entry, second entry) is $n/3$. Conversely, if the frequency of every row is $n/3$, then $F=\sum_{i=1}^n f_r(i)=n^2/3$.
\end{proof}

\newpage

\noindent\textbf{Further consequences}\\
We can extract more information from inequality (\ref{ineq:main}).

\smallskip

If $n=3s$ it gives $F\ge n^2/3$, with equality if and only if $m=s=n/3$ and $t=n-m=2m$.

\smallskip

If $n=3s+1$ it gives
\[F\ge \left\lceil\frac{9s^2+6s+1}{3}\right\rceil= \frac{9s^2+6s+3}{3}=\frac{n^2+2}{3},\]
with equality if and only if
\[\frac{(n-m-t)^2}{2}+\frac{(3s+1-3m)^2}{6}\le \frac23.\] The solutions to this are (i) $m=s$ and either $t=n-m$ or $t=n-m-1$, and (ii) $m=s+1$ and $t=n-m$. But case (ii) is impossible for $n>1$ since the minimum frequency $m=s+1$ cannot exceed the mean frequency $ \frac{n^2+2}{3n}=\frac{n}{3} + \frac{2}{3n}=s+\frac13+\frac{2}{3n}<s+1$ for $n>1$. So when $n=3s+1$ and $s\ne 0$, the only possibility for equality in (\ref{ineq:main}) is that $m=s$, and either $t=n-m=2m+1$ or $t=n-m-1=2m$. We will discuss these two options for $t$ below, and show that $t=2m+1$ cannot give a minimum MOPLS($n$).

\smallskip

If $n=3s+2$ inequality (\ref{ineq:main}) gives
\[F\ge \left\lceil\frac{9s^2+12s+4}{3}\right\rceil= \frac{9s^2+12s+6}{3}=\frac{n^2+2}{3},\]
with equality if and only if
\[\frac{(n-m-t)^2}{2}+\frac{(3s+2-3m)^2}{6}\le \frac23.\] The solutions to this are (i) $m=s$ and $t=n-m$, and (ii) $m=s+1$ and either $t=n-m$ or $t=n-m-1$. But case (ii) is impossible for $n>2$ since the minimum frequency $m=s+1$ cannot exceed the mean frequency $ \frac{n^2+2}{3n}=\frac{n}{3} + \frac{2}{3n}=s+\frac23+\frac{2}{3n}<s+1$ for $n>2$. So when $n=3s+2$ and $s\ne 0$, the only possibility for equality in (\ref{ineq:main}) is that $m=s$, and $t=n-m=2m+2$.

\bigskip

These observations, taken with the construction given in the Introduction, enable us to state the following corollary.

\begin{cor}\label{cor:main}
For $n\ge 21$, the minimum number of filled cells $F$ achievable in an MOPLS($n$) is given by
\[F=\left\{\begin{array}{cl}
n^2/3 &\mbox{if~~}  n\equiv 0 \pmod{3}\\
(n^2 + 2)/3 &\mbox{if~~} n\equiv 1 \mbox{~or~} 2 \pmod{3}
\end{array}\right\} = \left\lceil\frac{n^2}{3}\right\rceil.\]
Moreover, the minimum frequency in each of these cases is $m=\lfloor n/3 \rfloor$.
\end{cor}

It can also be shown that the minimum MOPLS($n$)s described in the Introduction are, essentially, the only minimum MOPLS($n$) for $n\ge 21$. The word ``essentially'' is necessary because it is possible to permute rows, columns and entries and thereby make these designs look different.

\begin{theorem}\label{th:subsquares}
Suppose that $n\ge 21$ and that $M$ is a \emph{minimum} MOPLS($n$), so that $M$ has $F=\left\lceil\frac{n^2}{3}\right\rceil$ filled cells. Then $M$ can be partitioned into three OLSs (i.e. mutually orthogonal Latin squares) on different point sets with orders that sum to $n$ and are either $\left\lfloor n/3\right\rfloor$ or $\left\lfloor n/3\right\rfloor+1$.
\end{theorem}

\begin{proof}
In order to prove this we start with some additional observations.

First, in the proof of Theorem \ref{th:main} we assumed (without loss of generality) that the minimum frequency $m$ of rows, columns, first entries, and second entries was achieved by a first entry $y_1$. This allows the possibility that the minimum frequency of rows, columns or second entries might be greater than $m$. But in the case of a minimum MOPLS($n$) with $F=\lceil n^2/3 \rceil$ filled cells (and $n\ge 3$) the average frequency is $F/n<\lfloor\frac{n}{3}\rfloor+1$, so there must be a row with frequency at most $\lfloor n/3 \rfloor=m$. A similar argument applies to columns and second entries. Hence, for a minimum MOPLS($n$) with $n\ge 21$, there must be a row, a column, a first entry, and a second entry with minimum frequency $m=\lfloor n/3 \rfloor$.

Returning to the argument employed in obtaining inequality  (\ref{ineq:main}), and referring to Figures \ref{fig:D} and \ref{fig:ABCD}, it can be seen that equality is only possible if all the following conditions hold:
\begin{enumerate}
 \item $f_r(i)=m$ and $f_c(i)=m$ for $i=1,2,\ldots,m$,
 \item $f_r(i)+f_c(i)=n-m$ for $i=m+1,m+2,\ldots,m+t$,
 \item $f_r(i)+f_c(j)=2(n-m)-t$ for $i,j=m+t+1,m+t+2,\ldots,n$.
\end{enumerate}

The cases $n\equiv 0,1,2$ (mod 3) will be examined individually. But in each case we may assume that $M$ is as represented in Figure \ref{fig:ABCD} with cell $(i,i)$ containing entry $y_1z_i$ for $i=1,2,\ldots,m$. Regions $B, C, D$ do not contain any first entry $y_1$. Region $D$ has a transversal $T$ of $t$ empty cells where
\[t=\left\{\begin{array}{lr} 2m & \mbox{~if~} n\equiv0 \pmod{3},\\ 2m \mbox{~or~} 2m+1 & \mbox{~if~} n\equiv1 \pmod{3}, \\2m+2 & \mbox{~if~} n\equiv2 \pmod{3}.
         \end{array}\right.\]
\newpage

\noindent\textbf{The case $n\equiv 0$ (mod 3)}

Suppose that $M$ is a minimum MOPLS($n$) with $n\equiv 0$ (mod 3) and $n\ge 21$. There are $m=n/3$ filled cells in each row and each column, and each first and each second entry appears $m$ times in $M$.

Choose any $z_k$ with $k>m$ and consider the possibility of placing the entry $y_1z_k$ in one of the $2m$ cells of $T$. Since this must be impossible, it follows that each cell of $T$ must have this second entry $z_k$ in its row or column in $M$. We will say that such an entry \emph{covers} the corresponding cell(s) of $T$. But an entry $z_k$ in region $B$ or $C$ covers only one cell of $T$ and there are $2m$ cells of $T$ to be covered by $m$ occurrences of $z_k$. So each occurrence of $z_k$ with $k>m$ must occur in region $D$, where it covers two cells of $T$. It follows that region $D$ must have $2m^2$ filled cells with each row and column within $D$ having $m$ filled cells. Consequently regions $B$ and $C$ have no filled cells, and region $A$ has the remaining $m^2$ filled cells. It then follows that all $m$ occurrences of $z_i$ for $i=1,2,\ldots,m$ must occur in region $A$ since $D$ is ``full'' with second entries $z_k$ for $k>m$. Hence $A$ forms a Latin square LS($m$) on the second entries $z_1, z_2, \ldots z_m$.

Next take the region $D$ and delete all first entries to form a PLS($2m$), say $D^*$, on the second entries $z_{m+1},z_{m+2},\ldots, z_n$. If it were possible to insert $z_i$ with $i>k$ into any empty cell of $D^*$, then it would be possible to insert the entry pair $y_1z_i$ into the corresponding cell of $M$. Since this is impossible, we may conclude that $D^*$ is an MPLS($2m$). Since it has $2m^2$ filled cells, by permuting rows and columns it can be put into the form described in Lemma \ref{lem:HR} and shown in Figure \ref{fig:HR}, with two LS($m$) subarrays on the leading diagonal on distinct point sets $\{z_{m+1},z_{m+2},\ldots, z_{2m}\}$ and $\{z_{2m+1},z_{2m+2},\ldots, z_{3m}\}$. The same permutations can be applied to region $D$ so that $D$ has two filled $m\times m$ subarrays on its leading diagonal, and other cells of $D$ are empty.

Swapping the roles of first and second coordinates in the argument above shows that region $A$ forms a LS($m$) on a first entry set, say $\{y_1, y_2, \ldots, y_m\}$. Consequently $A$ is an OLS($m$). The two $m\times m$ subarrays in region $D$ also form LS($m$)s on first entry sets, say $\{y_{m+1},y_{m+2},\ldots, y_{2m}\}$ and $\{y_{2m+1},y_{2m+2},$ $\ldots, y_{3m}\}$. Consequently $D$ has two OLS($m$)s on its leading diagonal and, outside the three OLS($m$)s, all other cells of $M$ are empty.

\bigskip

\noindent\textbf{The case $n\equiv 1$ (mod 3)}

Suppose that $M$ is a minimum MOPLS($n$) with $n\equiv 1$ (mod 3) and $n\ge 22$. The minimum frequency   is $m=(n-1)/3$ and $F=(n^2+2)/3=3m^2+2m+1$.

If the frequency of each second entry $z_i$ with $i\ge m+1$ is at least $m+1$, then $F\ge m^2 +(2m+1)(m+1)=3m^2+3m+1>3m^2+2m+1=F$, a contradiction. Hence there is at least one second entry, say $z_{m+1}$, that occurs $m$ times in $M$. Each cell of $T$ must be covered by $z_{m+1}$ in its row or column otherwise $y_1z_{m+1}$ can be inserted into an empty cell of $T$, a contradiction. Each entry $z_{m+1}$ can cover at most two cells of $D$, and then only if the occurrences of $z_{m+1}$ all occur in $D$. So at most $2m$ cells of $T$ can be covered by an entry $z_{m+1}$, and consequently $t\ne 2m+1$. It follows that $t=2m$ and that all $m$ second entries $z_{m+1}$ are in region $D$. It also follows that inequality (\ref{ineq:main}) is actually an equality, so that there are exactly $m$ entries in each row $i$ and each column $i$ for $i=1,2,\ldots, m$, there are $n-m=2m+1$ entries in the union of each row and column of $T$, and there are $2(n-m)-t-1=2m+1$ entries in the union of the row and column through the single (filled) cell representing $T'$. The rows and columns intersecting region $D$ can therefore be permuted to give $M$ in the form shown in Figure \ref{fig:1mod3} where it is subdivided into regions $M_i$ for $i=1,2,\ldots,9$.

\begin{figure}[!ht]\begin{center}
\begin{tikzpicture}[x=1mm,y=1mm]

\draw (10,10)--(10,110)--(110,110)--(110,10)--(10,10);
\draw [dotted] (10,80)--(110,80);
\draw [dotted] (10,50)--(110,50);
\draw [dotted] (40,10)--(40,110);
\draw [dotted] (70,10)--(70,110);

\node at (30,105) {$M_1=A$};
\node at (60,105) {$M_2$};
\node at (95,105) {$M_3$};
\node at (30,71) {$M_4$};
\node at (60,71) {$M_5$};
\node at (95,71) {$M_6$};
\node at (30,40) {$M_7$};
\node at (60,40) {$M_8$};
\node at (95,40) {$M_9$};

\node at (14,107) {$y_1z_1$};
\node at (25,95) {$\ddots$};
\node at (35,83) {$y_1z_m$};

\node at (49,77) {$y_{m+1}z_{m+1}$};
\node at (55,65) {$\ddots$};
\node at (62,53) {$y_{2m}z_{m+1}$};

\draw [Latex-Latex] (10,8)--(39,8);
\node at (25,3) {$m$};
\draw [Latex-Latex] (41,8)--(69,8);
\node at (55,3) {$m$};
\draw [Latex-Latex] (71,8)--(110,8);
\node at (90,3) {$m+1$};

\draw [Latex-Latex] (113,81)--(113,109);
\node at (117,95) {$m$};
\draw [Latex-Latex] (113,51)--(113,79);
\node at (117,65) {$m$};
\draw [Latex-Latex] (113,11)--(113,49);
\node at (120,30) {$m+1$};
\end{tikzpicture}
\end{center}
\caption{A partition of $M$ when $n\equiv 1$ (mod 3).\label{fig:1mod3}}
\end{figure}
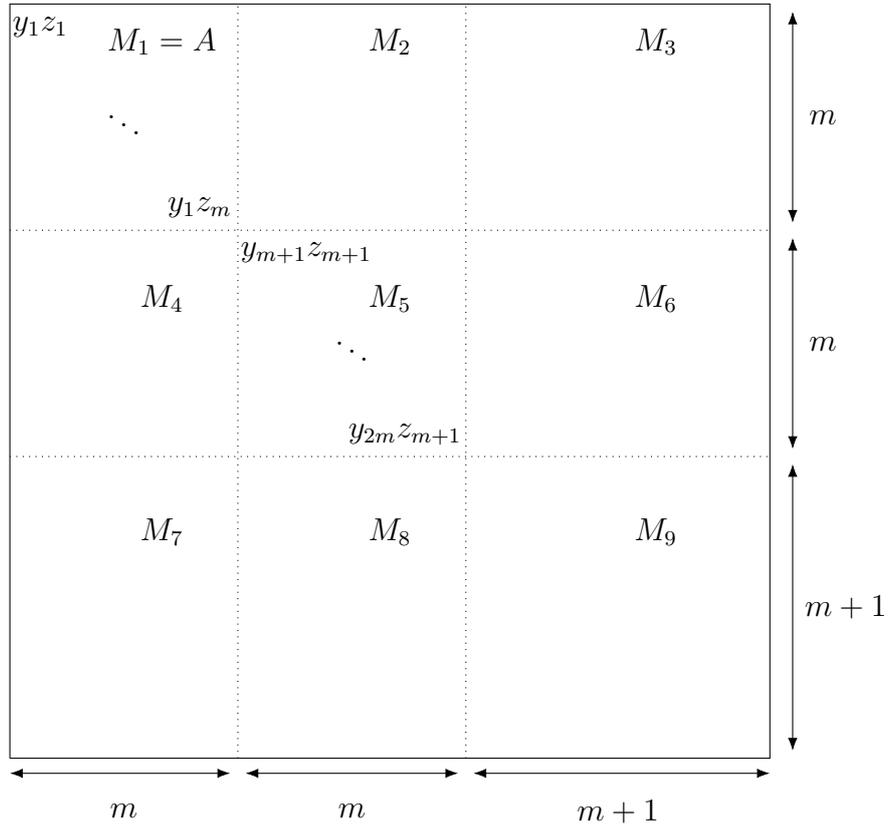

Relating Figure \ref{fig:1mod3} back to Figure \ref{fig:ABCD}, region $M_1$ is the old region $A$ with diagonal entries $y_1z_1,y_1z_2,\ldots,y_1z_m$, $M_2\cup M_3=B$, $M_4\cup M_7=C$, and $M_5 \cup M_6\cup M_8\cup M_9=D$. Region $M_5$ has diagonal entries $y_{m+1}z_{m+1},y_{m+2}z_{m+1},$ $ \ldots,y_{2m}z_{m+1}$.

Region $M_9$ must be full, otherwise it would be possible to insert $y_1z_{m+1}$ into an empty cell. So in $M$ there are at least $2m$ rows with $m$ entries and at least $m+1$ rows with $m+1$ entries, giving at least $2m^2+(m+1)^2=3m^2+2m+1$ entries. But this is the number of entries, so we must have exactly $2m$ rows with $m$ entries and $m+1$ rows with $m+1$ entries. By symmetry between rows, columns, first entries and second entries, the same applies to column frequencies, to first entry frequencies and to second entry frequencies. Since region $M_9$ covers $m+1$ rows and $m+1$ columns, it follows that regions $M_3,M_6,M_7$ and $M_8$ must all be empty.

Let $S$ denote the set of $m+1$ second entries occurring in column $2m+1$ of $M$, all of which occur in region $M_9$. Since $M$ is maximal, each row $m+i$ of $M$ for $1\le i\le m$ must contain all the second entries in $S'=\{z_{m+1}, z_{m+2},\ldots,z_{n}\}\setminus S$, otherwise there exists some $k$ with $m+1\le k \le n$, such that the entry pair $y_1z_k$ can be placed in cell $(m+i,2m+1)$. Consequently each element of $S'$ must occur in each of the rows $m+1,m+2,\ldots,2m$ of $M$.
Since $|S|=m+1$, we have $|S'|\ge m$, but because each row $m+i$ ($1\le i\le m$) has only $m$ entries, we must have $|S'|=m$, and it follows that $S$ cannot contain any of $z_1, z_2,\ldots,z_m$. Without loss of generality we may take $S=\{z_{2m+1},z_{2m+2}, \ldots, z_n\}$ and $S'=\{z_{m+1},z_{m+2},\ldots,z_{2m}\}$.

Next consider row $m+1$ of $M$: this contains all $m$ second entries in $S'$. By employing the same argument again it can be seen that each column $2m+i$ of $M$ for $1\le i \le m+1$ must contain all the second entries in $S$. It follows that all occurrences of each $z_k\in S$ lie in region $M_9$, so that $M_9$ forms a LS($m+1$) on the set $S$.

Region $M_5$ must be full, otherwise the entry pair $y_1z_n$ could be placed in an empty cell. It follows that regions $M_2$ and $M_4$ must be empty and that $M_5$ forms a LS($m$) on the set $S'$. This in turn implies that all occurrences of the second entries $z_1,z_2,\ldots, z_m$ are in region $M_1$ and hence that $M_1$ forms a LS($m$) on these second entries.

By repeating the argument of the previous three paragraphs with the roles of first and second entries reversed it can be seen that region $M_9$ forms a Latin square on first entries $y_{2m+1}, y_{2m+2}, \ldots,y_n$, region $M_5$ forms a Latin square on first entries $y_{m+1}, y_{m+2}, \ldots,y_{2m}$, and region $M_1$ forms a Latin square on first entries $y_{1}, y_{2}, \ldots,y_m$.  Consequently $M_1$ and $M_5$ are OLS($m$)s, and $M_9$ is an OLS($m+1$), the three subsystems are on different entry sets, and cells outside these regions are empty.

\newpage

\noindent\textbf{The case $n\equiv 2$ (mod 3)}

Suppose that $M$ is a minimum MOPLS($n$) with $n\equiv 2$ (mod 3) and $n\ge 23$. The minimum frequency   is $m=(n-2)/3$ and $F=(n^2+2)/3=3m^2+4m+2$.

Choose any $z_k$ with $k>m$ and consider the possibility of placing the entry $y_1z_k$ in one of the $2m+2$ cells of $T$. Since this must be impossible, it follows that each cell of $T$ must have a second entry $z_k$ in its row or column in $M$. An entry $z_k$ in regions $B$ or $C$ covers only one cell of $T$ and there are $2m+2$ cells of $T$ to be covered by occurrences of $z_k$. So there must be at least $m+1$ occurrences of $z_k$ with $k>m$, and this is only sufficient if all occurrences lie in region $D$ where each entry covers two cells of $T$. This implies that there are at least $m^2+(2m+2)(m+1)=3m^2+4m+2$ entries in $M$. But this is the number of entries, so we must have exactly $m+1$ occurrences of each $z_k$ with $k>m$, and each of these occurrences lies in region $D$. It follows that region $D$ must have $2(m+1)^2$ filled cells with each row and column within $D$ having $m+1$ filled cells. Consequently regions $B$ and $C$ have no filled cells, and region $A$ has the remaining $m^2$ filled cells. It then follows that all $m$ occurrences of $z_i$ for $i=1,2,\ldots,m$ must occur in region $A$ since $D$ is ``full'' with second entries $z_k$ for $k>m$. Hence $A$ forms a Latin square LS($m$) on the second entries $z_1, z_2, \ldots z_m$.

Next take the region $D$ and delete all first entries to form a PLS($2m+2$), say $D^*$, on the second entries $z_{m+1},z_{m+2},\ldots, z_n$. If it were possible to insert $z_i$ with $i>k$ into any vacant cell of $D^*$, then it would be possible to insert the entry pair $y_1z_i$ into the corresponding cell of $M$. Since this is impossible, we may conclude that $D^*$ is an MPLS($2m+2$). Since it has \linebreak $(2m+2)^2$ filled cells, by permuting rows and columns it can be put into the form described in Lemma \ref{lem:HR} and shown in Figure \ref{fig:HR}, with two LS($m+1$) subarrays on the leading diagonal on distinct point sets $\{z_{m+1},z_{m+2},\ldots, z_{2m+1}\}$ and $\{z_{2m+2},z_{2m+3},\ldots, z_{3m+2}\}$. The same permutations can be applied to region $D$ so that $D$ has two filled $(m+1)\times (m+1)$ subarrays on its leading diagonal, and other cells of $D$ are empty.

Swapping the roles of first and second coordinates in the argument above shows that region $A$ forms a LS($m$) on a first entry set, say $\{y_1, y_2, \ldots, y_m\}$. Consequently $A$ is an OLS($m$). The two $(m+1)\times (m+1)$ subarrays in region $D$ form LS($m+1$)s on first entry sets, say $\{y_{m+1},y_{m+2},\ldots,$ $ y_{2m+1}\}$ and $\{y_{2m+2},y_{2m+3},\ldots, y_{3m+2}\}$. Consequently $D$ has two OLS($m+1$)s on its leading diagonal and, outside the three OLSs, all other cells of $M$ are empty.

This completes the proof of Theorem \ref{th:subsquares}.
\end{proof}

\newpage

\noindent\textbf{Connection with coding theory}

As explained in \cite{MOPLS1}, there is a close connection between MOPLS($n$)s and $n$-ary codes of length 4 having minimum distance 3 and covering radius 2. If $M$ is an MOPLS($n$), then $M$ gives rise to an $n$-ary code $C$ of length 4 by taking as codewords the (row, column, first entry, second entry) quadruples corresponding to filled cells of $M$. Two codewords of $C$ can agree in at most one coordinate, so $C$ has minimum distance 3 or 4. But minimum distance 4 can only occur if each row and each column of $M$ has at most one entry pair, and each first entry symbol and each second entry symbol appears at most once in $M$. So for $n>3$, the minimum distance of the code $C$ will be 3.

Suppose that the code $C$ is derived from an MOPLS($n$), say $M$, as described. Since no empty cell of $M$ can be filled, every quadruple not in $C$ must lie within distance 2 of a codeword of $C$. In other words, the covering radius $\rho$ of $C$ is at most 2. For $n\ge 2$ we cannot have $\rho=0$ as this would imply that every possible quadruple is a codeword, giving $n^4\le n^2$. If $\rho=1$ then $C$ is a perfect code, but the only perfect code  of length 4 and minimum distance 3 is the ternary Hamming code Ham($2,3$) with 9 codewords, and this corresponds to an OLS(3). If $C$ is not a perfect code then $\rho=2$.

Hence, for $n>3$ an MOPLS($n$) is equivalent to an $n$-ary code of length 4 having minimum distance 3 and covering radius 2. And it follows from our Theorem \ref{th:main} and Corollary \ref{cor:main} that for $n\ge 21$ the minimum number of codewords in such a code is $\lceil n^2/3 \rceil$.

\section{Concluding Remarks}
This paper resolves the main conjecture made in \cite{MOPLS1}. However, a further conjecture was made in that paper concerning $k$-MOPLS. A set of $k$ partial Latin squares of the same order $n$, having the same filled cells as one another, are said to be \emph{orthogonal} if, when superimposed, the (row, column, entry 1, entry 2, $\ldots$, entry $k$) $(k+2)$-tuples agree in at most one coordinate position. We use the superimposed form and denote such an array as a $k$-OPLS($n$). Alternative representations are as a partial transversal design PTD($k+2,n$) or as a decomposition of a subgraph of $K_{(k+2)\times n}$ into copies of $K_{k+2}$.

A \emph{maximal} $k$-OPLS($n$) is a $k$-OPLS($n$) that cannot be extended to another $k$-OPLS($n$) by inserting any entry $k$-tuple into any empty cell. We denote a maximal $k$-OPLS($n$) as a $k$-MOPLS($n$).  Given $k$, for each sufficiently large $n$ there is a set of $k$ mutually orthogonal Latin squares of order $n$ (see \cite[page 163]{HBK}), and these form a maximal $k$-MOPLS of maximum size, i.e. having $n^2$ filled cells. We will denote such a design (in superimposed form) as a $k$-OLS($n$). We are concerned with the minimum size (number of filled cells) of any $k$-MOPLS($n$).

An example of a 3-MOPLS(16) is displayed in Figure \ref{fig:3-MOPLS(16)}. This has 64 filled cells and 192 empty cells. It is formed from four copies of a set of three mutually orthogonal Latin squares of order 4 (i.e. four 3-OLS(4)s), placed down the leading diagonal.

\begin{figure}[ht]
\begin{center}
\setlength\tabcolsep{1.5mm}
\tiny\begin{tabular}{|cccccccccccccccc|}\hline
                \rule{0mm}{3mm}111&222&333&444&-&-&-&-&-&-&-&-&-&-&-&-\\
                432&341&214&123&-&-&-&-&-&-&-&-&-&-&-&-\\
                243&134&421&312&-&-&-&-&-&-&-&-&-&-&-&-\\
                324&413&142&231&-&-&-&-&-&-&-&-&-&-&-&-\\
                -&-&-&-&555&666&777&888&-&-&-&-&-&-&-&-\\
                -&-&-&-&876&785&658&567&-&-&-&-&-&-&-&-\\
                -&-&-&-&687&578&865&756&-&-&-&-&-&-&-&-\\
                -&-&-&-&768&857&586&675&-&-&-&-&-&-&-&-\\
                -&-&-&-&-&-&-&-&999&AAA&BBB&CCC&-&-&-&-\\
                -&-&-&-&-&-&-&-&CBA&BC9&A9C&9AB&-&-&-&-\\
                -&-&-&-&-&-&-&-&ACB&9BC&CA9&B9A&-&-&-&-\\
                -&-&-&-&-&-&-&-&BAC&C9B&9CA&AB9&-&-&-&-\\
                -&-&-&-&-&-&-&-&-&-&-&-&DDD&EEE&FFF&GGG\\
                -&-&-&-&-&-&-&-&-&-&-&-&GFE&FGD&EDG&DEF\\
                -&-&-&-&-&-&-&-&-&-&-&-&EGF&DFG&GED&FDE\\
                -&-&-&-&-&-&-&-&-&-&-&-&FEG&GDF&DGE&EFD\\
                \hline\end{tabular}
                \caption{A 3-MOPLS(16)}\label{fig:3-MOPLS(16)}
                \end{center}
\end{figure}

It was shown in \cite{MOPLS1} that this is indeed a maximal 3-MOPLS(16).  And it was explained how the construction may be generalised to give a \linebreak $k$-MOPLS($n$) having approximately $n^2/(k+1)$ filled cells, provided that $n$ is sufficiently large. This is done by placing suitable $k$-OLS designs (all on different entry sets) down the leading diagonal of an otherwise empty $n\times n$ array. This led to the conjecture that this would give the minimum possible number of filled cells: $F\approx n^2/(k+1)$.

Again there is a connection with coding theory. If $M$ is a $k$-MOPLS($n$) then the code $C$ formed from the $(k+2)$-tuples has minimum distance\linebreak $k+1$ and covering radius $\rho\le k$. The code associated with the construction described briefly above has covering radius $\rho=k$.

\end{document}